\documentclass[11pt,a4paper,leqno]{article}
\usepackage{amsmath,amsthm,amssymb,amsfonts,color,graphicx}

\newcommand{\C}{{\mathbb{C}}}          
\newcommand{\R}{{\mathbb{R}}}          
\newcommand{\Proj}{{\mathbb{P}}}       

\newcommand{\sol}{{\mathfrak{o}}}
\newcommand{\sul}{{\mathfrak{su}}}
\newcommand{\uni}{{\mathfrak{u}}}

\newcommand{\rr}{\rightarrow}
\newcommand{\lrr}{\longrightarrow}

\newcommand{\calT}{{\cal T}}

\newcommand{\calO}{{\cal O}}
\newcommand{\calri}{{{\cal R}^\xi}} 
\newcommand{\tm}{{{\cal T}_{M}}}

\newcommand{\SO}{{\mathrm{SO}}}
\newcommand{\SU}{{\mathrm{SU}}}

\newcommand{\na}{{\nabla}}
\newcommand{\dx}{{\mathrm{d}}}
\newcommand{\inv}[1]{{#1}^{-1}}
\newcommand{\papa}[2]{\frac{\partial#1}{\partial#2}}
\newcommand{\vol}{{\mathrm{vol}}}
\newcommand{\partialoverlinez}{{\partial_{\overline{z}}}}
\newcommand{\partialoverlinew}{{\partial_{\overline{w}}}}

\newcommand{\estrela}{{\boldsymbol{\star}}}

\newcommand{\cinf}[1]{{\mathrm{C}}^\infty_{#1}}

\newcommand{\mg}[1]{{g_{_{#1}}}}

\newtheorem{teo}{\textit{THEOREM}}[section]

\newtheorem{prop}{\textit{PROPOSITION}}[section]

\title{AN INVARIANT K\"AHLER METRIC ON THE TANGENT DISK BUNDLE OF A SPACE-FORM}

\author{R. ALBUQUERQUE}

\begin{document}

\maketitle

\markright{\sl\hfill {\small{A K\"ahler metric on the tangent disk bundle of a space-form}}     \hfill}

\vspace*{2mm}

\begin{abstract}

We find a family of K\"ahler metrics invariantly defined on the radius $r_0>0$ tangent disk bundle ${\calT_{M,r_0}}$ of any given real space-form $M$ or any of its quotients by discrete groups of isometries. Such metrics are complete in the non-negative curvature case and non-complete in the negative curvature case. If $\dim M=2$ and $M$ has constant sectional curvature $K\neq0$, then the K\"ahler manifolds ${\calT_{M,r_0}}$ have holonomy $\SU(2)$; hence they are Ricci-flat. For $M=S^2$, just this dimension, the metric coincides with the Stenzel metric on the tangent manifold ${\calT_{S^2}}$, giving us a new most natural description of this well-know metric.

\end{abstract}

\vspace*{2mm}


\vspace*{1mm}


\vspace{7mm}
\begin{center}
 \textbf{1. INTRODUCTION}
\end{center}
\setcounter{section}{1}

We here bring to light the remarkable Hermitian structure $\mg{\tm},\omega,J$ existing on the space $\calT_{M,r_0}$ and which has a K\"ahler manifold structure. The real manifold is the radius $r_0$ open disk bundle contained in $\calT_M$, the total space of the tangent bundle $TM\lrr M$ of any given constant sectional curvature $K$ Riemannian manifold $M$. Recurring to the canonical horizontal-plus-vertical decomposition of $T\calT_M$, we have
\begin{equation}\label{theKahlermetric0}
 \mg{\calT_{M}}=\sqrt{c_1+Kr^2}\,\pi^*g_{_M}+\frac{1}{\sqrt{c_1+Kr^2}}\,\pi^\estrela g_{_{M}}  
\end{equation}
where $r=\|u\|_{_M},\ u\in\calT_M$, and $c_1$ is any positive constant, $r_0=\sqrt{-c_1/K}$ for $K<0$ and $r_0=+\infty$ for $K\geq0$. The general almost Hermitian structure induces a K\"ahler metric if and only if the conditions on the above weights and base are fulfilled. We then try to study the properties and associated questions of the metric. Conferring the vast literature on the geometry of tangent bundles, we find that it is not completely unaware of the structure. The present discovery is, in fact, a particular case of the K\"ahler metrics found by V. Oproiu and N. Papaghiuc in \cite{OproiuPapag}.

In the case of non-negative sectional curvature, we must have radius $r_0=+\infty$ in order to have a complete metric. In the negative case, in order to have the structure well-defined, $r_0$ must be finite and then it follows the former cannot be complete.

We prove that the Hermitian structures defined on $\calT_{M,r_0}$ are naturally preserved by the lift of any isometry of the base manifold $M$. Indeed the group of isometries of $M$ lifts to the group of automorphisms of $\mg{\tm},J$, and thus the K\"ahler metric on the tangent manifold becomes an intrinsic object of space-form geometries and any of their smooth Riemannian quotients.

Using results from \cite{Alb2014d} we also show the K\"ahler metric is non-flat and Ricci-flat if and only if $m=2$ and $K\neq0$. In other words, we find a non-compact Calabi-Yau metric or holonomy $\SU(2)$ manifold.

Let us emphasise that, for $K>0$ and $\dim M=2$, our metric is shown to be the well-known Stenzel metric on $\calT_{S^2}$. This does not hold in higher dimensions. Indeed, the metric is Ricci-flat if and only if $m=2$ and the Stenzel or Eguchi-Hanson metric on $\calT_{S^m}$ is always Ricci-flat. Our construction has the virtue of working with a non-orientable base or with a hyperbolic space $H^2$, taking the respective adjustments. We do have now explicit $\SU(2)$-holonomy metrics on $\calT_{\R\Proj^2}$ and $\calT_{H^2,r_0}$.

We remark that, from its general construction, the new K\"ahler manifold has none of the \textit{expected} properties of a holomorphic bundle. Both zero-section and fibres are, simultaneously, $\R$-Lagrangian submanifolds. We obtain a rare example of an invariantly defined fibre bundle K\"ahler structure satisfying those two properties together.

In the last chapter of the article, we endeavour to discover the complex charts, or just a \textit{totally commuting complex frame field}, which would let us write for instance the Ricci-form. The former notion is introduced here, for the study of the cases $\dim M=2$ and curvature $\pm1$. We conclude this problem must have a more analytic rather than a geometric approach.

There is also a {notable} similarity with the Bryant-Salamon $\mathrm{G}_2$ metric on $\Lambda^2_-T^∗S^4$
since its weight functions and those in \eqref{theKahlermetric0} are very much alike.

We follow notation and the theory introduced in previous works, such as \cite{Alb2008,Alb2012,Alb2014d}.

\vspace{3mm}
\begin{center}
 \textbf{2. THE TANGENT MANIFOLD WITH SASAKI-TYPE HERMITIAN METRIC}
\nopagebreak
\end{center}
\setcounter{section}{2}
A smooth linear connection $\na$ on a smooth manifold $M$ opens the way to define new smooth structures on the total space $\tm$ of the tangent bundle $\pi:TM\lrr M$, arising from the structures on $M$. Until the end of this article we assume $\na$ is torsion-free.

The canonical charts induced from an atlas of $M$ and the corresponding trivialisations of $TM$ show that the tangent sub-bundle to the fibration $(TM,\pi,M)$ agrees with the kernel $V$ of $\dx\pi$. In particular, we have $V_u=T_u(T_xM)=\{u\}\times T_{x}M$ where $x=\pi(u)$. The identification of $V$ with $\pi^*TM$ thus follows from the very nature of the tangent bundle of $M$. Furthermore we have an exact sequence of vector bundles over $\tm$
\begin{equation}
 0\lrr V\lrr T\tm\stackrel{\dx\pi}\lrr\pi^*TM\lrr 0.
\end{equation}

The linear connection $\na$ gives a canonical decomposition of the tangent bundle of $\tm$ into $T\tm=H^\na\oplus V$, where $H^\na$ depends on the connection. Since the restriction ${\dx\pi_u}_|:H^\na_u\rr T_xM$ is an isomorphism, $\forall x\in M,\ u\in\inv{\pi}(x)$, both sub-bundles $H^\na$ and $V$ are isomorphic to the vector bundle $\pi^*TM$. Knowing this canonical decomposition, we conclude the linear connection $\pi^*\na\oplus\pi^*\na$, denoted $\na^*$, is well-defined as a linear connection on the manifold $\tm$. Clearly, the canonical projections are parallel morphisms.

A tautological vector field $\xi$ over $\tm$ is defined by $\xi_u=u$, through the \textit{vertical} lift. The projection $w\mapsto w^v,\ \, w\in T\tm$, (same notation for lift and projection should not be confusing) coincides with the map $\pi^*\na_\cdot\xi:T\tm\lrr V$. By construction, the horizontal distribution $H^\na$ agrees with the kernel of this map. The following expresses a fundamental identity of the theory. The torsion of $\na^*$ is
\begin{equation}\label{fundamentalidentity}
 {\na^{*}}_wz-{\na^{*}}_zw-[w,z]=\calri(w,z)\ ,\ \ \forall w,z\in T\tm ,
\end{equation}
where $\calri(w,z)=\pi^*R(w,z)\xi$ and $R$ is the curvature tensor of $(M,\na)$. Notice $\calri$ is indeed a tensor, depends only on the horizontal parts of $z,w$, and assumes only vertical values. Using projections and $\na^*_w\xi=w^v$, turns the long proof of \eqref{fundamentalidentity} with charts into a more pleasant verification.

Now suppose we have a Riemannian manifold $(M,g_{_M})$. We let $m=\dim_\R M$ and denote by $\na$ the Levi-Civita connection. We may both pull-back to horizontals and lift to verticals the given symmetric tensor. We distinguish the two, respectively, by $\pi^*g_{_M}$ and $\pi^\estrela g_{_M}$.

Any two positive real smooth functions $\mu,\lambda$ defined on $\tm$ induce a Riemannian metric $\mg{\tm}$ on $\tm$ depending on the canonical decomposition. It is defined by
\begin{equation}\label{themetric}
 \mg{\tm}=\mu^2\pi^*g_{_M}+\lambda^2\pi^\estrela g_{_M} .
\end{equation}
The well-known original metric $\mg{\tm}_0$, the case when $\mu=\lambda=1$, is the Sasaki metric.

A weighted Hermitian structure on the same manifold arises then from a $\mg{\tm}$-compatible almost-complex structure $J$. At each point $u\in\tm$ one defines an endomorphism $J$ of $T_u\tm$ by permuting lifts and by the correspondence, $\forall w\in T_{\pi(u)}M$,
\begin{equation}
w^h\longmapsto \frac{\mu}{\lambda}w^v\longmapsto -w^h .
\end{equation}
This clearly generalises the case $\mu=\lambda=1$, denoted $J_0$, which is again due to Sasaki (\cite{Sasa}). We call the endomorphism $B:T\tm\lrr T\tm$ defined by $Bw^h=w^v$, $Bw^v=0$ the \textit{mirror} map. Arising from previous studies, the mirror map proves to be quite useful since it is indeed a tensor and a $\na^{*}$-parallel one. We have 
\begin{equation}\label{thecomplexstructure}
J=\frac{\mu}{\lambda}B-\frac{\lambda}{\mu}B^\mathrm{ad} 
\end{equation}
where the adjoint is taken with respect to $\mg{\tm}_0$. Now, the associated symplectic 2-forms $\omega=J\lrcorner\mg{\tm}$ and $\omega_0=J_0\lrcorner\mg{\tm}_0$ satisfy (cf. \cite[Proposition 6.1]{Alb2012})
\begin{equation}\label{expressaodeomega}
  \omega=\mu\lambda\,\omega_0 .
\end{equation}
Indeed, $\omega(w^v,z^h)=\mg{\tm}(Jw^v,z^h)=-\frac{\lambda}{\mu}\mu^2\,\pi^*g_{_M}(w^h,z^h)=\lambda\mu\,\omega_0(w^v,z^h)$, $\forall w,z\in TM$, and for other lifts of $w,z$ the computation is quite similar.

The Hermitian structure $(J_0,\mg{\tm}_0)$ plays an important role. Until the end of the article the musical isomorphism $\cdot^\flat:T\tm\lrr T^*\tm$ shall refer to the Sasaki metric $\mg{\tm}_0$. And we assume always $m=\dim_\R M=\dim_\C\tm>1$.
\begin{teo}\label{TeoremaTMsymplectic}
 The non-degenerate 2-form $\omega$ is closed if and only if $\mu\lambda$ is a constant.
\end{teo}
\begin{proof}
Consider the 1-form on $\tm$ defined by $\theta(w)=\xi^\flat(Bw)$. Correcting $\na^*$ to a torsion-free connection $D^*=\na^*-\frac{1}{2}\calri$ and recalling $\calri$ takes values in $V$, we find
\begin{align*}
\dx\theta(w,z)&= (D^*_w\theta)z-(D^*_z\theta)w \\
&=  w(\theta(z))-\theta(\na^*_wz)-z(\theta(w))+\theta(\na^*_zw) \\
&= \mg{\tm}_0(w^v,Bz)-\mg{\tm}_0(z^v,Bw) \\
&= -\omega_0(w,z)
\end{align*}
Now clearly $\dx\omega=\dx(\mu\lambda)\wedge\omega_0$. And since the wedge of 1-forms with a symplectic form is injective, the result follows.
\end{proof}
Here follows a heuristic confirmation of the above. Notice the map $\cdot^\flat:\tm\rr{\cal T}^*_M,\ v\mapsto v^\flat$, is a \emph{diffeomorphism}, invariant for the respective induced connections $\na^*$ arising from the Levi-Civita connection. Simply, because on horizontals it is the identity and on verticals the canonical map $\cdot^\flat:TM\rr T^*M$ permutes the respective connections. Then one checks that $\theta$ is the pull-back by $\cdot^\flat$ of the canonical Liouville 1-form $\ell$ on the cotangent manifold. Moreover, as it is well-known, $-\dx\ell$ is an exact symplectic form on ${\cal T}^*_M$, so the same happens with the pull-back $\omega_0=-\dx\theta$.

In this article we assume the smooth functions $\mu,\lambda$ are dependent only of
\begin{equation}
 r=r(u)=\|u\|_{_M},\ \  u\in\tm,
\end{equation}
and are con\-ti\-nu\-ous\-ly differentiable at $0$. Since $r^2=\pi^\estrela g_{_M}(\xi,\xi)$, we easily apply the $\mg{\tm}_0$ metric connection $\na^{*}$ in order to find
\begin{equation}\label{derivativeofr}
 \dx r^2=2\xi^\flat .
\end{equation}

Let us now assume we have an isometry $f:M\lrr M$.
\begin{prop}\label{isometries}
 Under the above conditions, the differential $\dx f:TM\lrr f^*TM$ is a vector-bundle isomorphism, which corresponds to a well-defined iso\-me\-try $f_\star:\tm\lrr\tm$ and this manifold isometry is pseudo-holomorphic. In other words, the Hermitian structure $\mg{\tm},J$ is invariant under isometries $f$ of $M$.
\end{prop}
\begin{proof}
 We define $f_\star(u)=\dx f_{\pi(u)}(u)\in T_{f(\pi(u))}M$ and hence the differential of $f_\star$ coincides with the map $\dx f$ on vertical directions. Due to uniqueness, the Levi-Civita connection is invariant by the diffeomorphism $f$ and so $H^\na$ is preserved by $\dx f$. Hence $\dx f_\star$, up to conjugation by $\dx\pi_{|H^\na}$, also coincides with $\dx f$ on horizontal directions. Therefore $\mg{\tm}_0$ is preserved by $f_\star$. Moreover, the weighted metric $\mg{\tm}$ with functions $\lambda,\mu$, functions of $r$ only, is preserved by $f_\star$. Finally, since the horizontal and the vertical distributions are both preserved, it becomes easy to see that 
 \[ J_{f_\star(u)}\circ(\dx f_\star)_u=(\dx f_\star)_u\circ J_u,\ \ \forall u\in\tm.  \]
 In other words, $f_\star$ is pseudo-holomorphic.
\end{proof}
Many functorial properties from the action of isometry groups on $M$ thus carry over to the Hermitian metric on $\tm$.

\vspace{3mm}
\begin{center}
 \textbf{3. INTEGRABILITY OF $J$}
\end{center}
\setcounter{section}{3}\label{sectionIntegrabi}

For each $r_0>0$ let us denote by $\calT_{M,r_0}\subset\tm$ the open disk (bundle) of radius ${r_0}$:
 \begin{equation}
  \calT_{M,r_0}=\{u\in \tm:\ \|u\|<r_0\} .
 \end{equation}
We may allow the functions $\mu,\lambda$ to be only defined on a same interval of $\R^+_0$. The following Theorem is independent but concerns with a particular case in \cite{OproiuPapag}. We refer the reader also to this reference for the extensive literature on the geometry of tangent bundles.
\begin{teo}\label{IntegrabilidadedeSasakigenemdiscobundle}
 The largest disk bundle $\calT_{M,r_0}$ where the almost-complex structure $J$ is defined and integrable is obtained with, and only with, the following data:
 \begin{itemize}                                                                                                                                                                                                                                                                                                                                                                      \item $M$ has constant sectional curvature $K$
 \item $\frac{\mu}{\lambda}(r)=\sqrt{c_1+Kr^2}$, \,$\forall r\in[0,r_0[$, where  $c_1>0$ is any constant
 \item $r_0=\sqrt{-\frac{c_1}{K}}$ for $K<0$ and $r_0=+\infty$ for $K\geq0$.
 \end{itemize}
\end{teo}
\begin{proof}
Let us write $J=aB-\inv{a}B^{\mathrm{ad}}$ where $a=\frac{\mu}{\lambda}$. Then a first proof of the analytic integrability is to take, in the notation of \cite[Theorem 3]{OproiuPapag}, the values $a_1=\inv{a_2}=a$, $a_3=a_4=b_1=\cdots=b_4=0$ and apply that result.

But here is a second proof obtained easily with our experimented techniques.
By the Theorem of Newlander-Nirenberg, the integrability of $J$ is equivalent to the vanishing of the Nijenhuis tensor $N(w,z)=J[w,Jz]+J[Jw,z]+[w,z]-[Jw,Jz]$ for all tangent vectors at any given point of $\tm$. Since the tensor $N$ is complex anti-linear, it follows that it is enough to see $N(w,z)=0$ with $w,z\in{H}^\na$. Moreover, by tensoriality, we may just take lifts of vector fields of $M$. We have the formula $\dx a=a'\dx(r^2)=2a'\xi^\flat$ where $a'$ is the derivative with respect to $r^2$. Hence $\dx a(w)=0$ and $\dx a(Bw)=2a'\xi^\flat(Bw)$, $\forall w\in H^\na$. Now, using the torsion-free connection $\na^*-\frac{1}{2}\calri$, we find
 \begin{eqnarray*}
   N(w,z)&=& J[w,aBz]+J[aBw,z]+[w,z]-[aBw,aBz]  \\
   &=& J\bigl(\na^*_w(aBz)-\na^*_{aBz}w\bigr)         +J\bigl(\na^*_{aBw}z-\na^*_z(aBw)\bigr)+ \\
   & &\ \ \na^*_wz-\na^*_zw-\calri(w,z)-\na^*_{aBw}(aBz)+\na^*_{aBz}(aBw) \\
   &=& -\inv{a}\dx a(w)z-\na^*_wz-a^2\na^*_{Bz}Bw+a^2\na^*_{Bw}Bz+\inv{a}\dx a(z)w+\na^*_zw +\na^*_wz\\
   & & \ \ -\na^*_zw-\calri(w,z)-a\dx a(Bw)Bz-a^2\na^*_{Bw}Bz+a\dx a(Bz)Bw+a^2\na^*_{Bz}Bw  \\
   &=& -\pi^*R(w,z)\xi+2aa'\bigl(\xi^\flat(Bz)Bw-\xi^\flat(Bw)Bz\bigr) .
 \end{eqnarray*}
At the base level, integrability is thus equivalent to $R(w,z)u=2aa'(\langle z,u\rangle_{_M}w-\langle w,u\rangle_{_M}z)$. Since $a$ is a function only of $r$, even in dimension 2 we must have constant sectional curvature, say $K$. Integrating, we find $a^2=c_1+Kr^2$. Now the boundary conditions follow.
\end{proof}
Notice the flat base case ($K=0$) is not as trivial as that easily deduced for $J_0$. Indeed, the condition $\mu/\lambda$ constant may give holonomy equal to $\sol(2m)$, cf. \cite[Proposition 2.1.i]{Alb2014d}. However, if $\mu,\lambda$ are constants, then we do have a flat manifold $\tm$, by  \cite[Proposition 2.1.iii]{Alb2014d}, which is K\"ahler, by Theorem \ref{TeoremaTMsymplectic}.

Suppose $M$ has constant sectional curvature $K$. Combining Theorems  \ref{TeoremaTMsymplectic},\ref{IntegrabilidadedeSasakigenemdiscobundle}, we obtain the invariant K\"ahler metric on $\calT_{M,r_0}$ by solving $\mu/\lambda=\sqrt{c_1+Kr^2}$, $\mu\lambda=c_2>0$ constant. This yields
\begin{equation}
 \mu=\frac{c_2}{\lambda}=\sqrt{c_2}\sqrt[4]{c_1+Kr^2}.
\end{equation}
Up to a global constant conformal factor the desired metric is
\begin{equation}\label{theKahlermetric}
 \mg{\calT_{M,r_0}}=\sqrt{c_1+Kr^2}\,\pi^*g_{_M}+\frac{1}{\sqrt{c_1+Kr^2}}\,\pi^\estrela g_{_{M}}  .
\end{equation}
\begin{teo}
 The K\"ahler metric $\mg{\calT_{M,r_0}}$ is complete if $K\geq0$ and non-complete if $K<0$.
\end{teo}
\begin{proof}
 It follows by straightforward computations that the fibres are totally geodesic, cf. \cite[Proposition 1.3]{Alb2014d}. Suppose $\gamma$ denotes a fibre-ray which issues from 0 to the boundary of $\tm$. The length of the linear curve $\gamma$ is $l(\gamma)=\int_0^{{r_0}}\frac{1}{\sqrt[4]{c_1+Kt^2}}\dx t$. Then $l(\gamma)=+\infty$ in the case $K\geq0,\ r_0=+\infty$, and is finite in the negative curvature case because the integral converges on $[0,r_0[$ where $r_0=\sqrt{-\frac{c_1}{K}}$. Finally, any geodesic may be indefinitely extended in the first case, whereas in the second some geodesics cannot.
\end{proof}

Improving on the above results, we may also consider the case of the complementary space $\tm\backslash\overline{\calT_{M,r_0}}$ when $K>0$, $c_1\leq0$ and $r_0=\sqrt{-\frac{c_1}{K}}$. However, we find no complete metric here as well.

Since all structures $\mg{\tm}$, $J$ and $\omega$ are invariant under the lift of isometries of $M$, by Proposition \ref{isometries}, our K\"ahler metrics may also be found on smooth quotients of a same invariant construction over the tangent manifold of any given real space-form.

Finally, the metric \eqref{theKahlermetric} has a remarkable resemblance with the well-known $\mathrm{G}_2$-holonomy metric of Bryant-Salamon on the manifolds $\Lambda^2_\pm T^*M^4$. The weight functions are formally the same, with constant scalar curvature replacing $K$, and the completeness issues are analogous (cf. \cite{Alb2014d} and the references therein).

\vspace{3mm}
\begin{center}
\textbf{4. FURTHER GEOMETRIC PROPERTIES}
\end{center}
\setcounter{section}{4}

A first problem with the K\"ahler metric found above is to compute the volume of $\calT_{M,r_0}$ for any $K,r_0$. This becomes quite easy by simply recalling from \eqref{expressaodeomega} that $\omega=\omega_0$. Then we may apply the co-area formulae \cite[p. 125, 160]{Chavel} to deduce:
\begin{equation}
\vol(\calT_{M,r_0})=\vol_{g_{_M}}(M)\vol_{_{\R^m}}(D_{r_0}(0))=\vol_{g_{_M}}(M)\frac{\pi^{m/2}r_0^m}{\Gamma(m/2+1)}.
\end{equation}

We may normalise the metric \eqref{theKahlermetric} within the constant values of $K,c_1$; indeed these correspond either to a conformal factor of $g_{_M}$ on $M$ and/or to conformal factors of the metric in $H^\na$ and $V$. Yet, they are helpful in computations.

The restriction of $\pi$ to $\tm$ is not a Riemannian submersion in general. Some curvature related results for $\mg{\tm}$ may be easily obtained from the theory in \cite{Alb2014d}, such as the Levi-Civita connection, the curvature on the zero-section and on the fibres. The fibres are totally geodesic Riemannian submanifolds. From \cite[Section 1.5]{Alb2014d} we find that the $m$-plane disk $D_{r_0}(0)\subset\R^m$ with spherically symmetric metric $\lambda^2\sum_{i=1}^m(\dx y^i)^2$ has scalar curvature
\begin{equation}
 \mathrm{Scal}=
 \frac{(m-1)K}{4(c_1+Kr^2)^\frac{3}{2}}(3(m-2)Kr^2+4mc_1) .
\end{equation}
In case $m=2$, this is $2Kc_1/(c_1+Kr^2)^\frac{3}{2}$. And therefore one sees that, in the boundary, when $r\nearrow r_0$, the scalar curvature is $0$, if $K>0$, and $-\infty$, if $K<0$. Also notice the sectional curvature of the zero-section $M\hookrightarrow\tm$, also totally geodesic, is clearly $K/\sqrt{c_1}$. But this is precisely the sectional curvature of the fibre $D_{r_0}$ at 0, so we may wonder of an Einstein metric on $\tm$ (we recall, for Einstein metrics on 4-manifolds, orthogonal planes have the same sectional curvature).

In searching for Einstein metrics on the whole space, there is a test in \cite[Corollary 2.1]{Alb2014d} which one may perform when the base manifold is itself Einstein. Again, like the $\mathrm{G}_2$ metrics of Bryant-Salamon, our K\"ahler manifolds satisfy that necessary condition, from the Ricci tensor, and indeed they may well be Einstein, with Einstein constant 
\begin{equation}
 (m-2)K/2\sqrt{c_1}.
\end{equation}
It is easy to apply the referred test. Certainly, when $m\neq2$, the spaces are not Ricci flat and thus their holonomy is ${\mathrm{U}}(m)$. Let us give an explicit proof of this result.
\begin{teo}
 In case $K=0$, the K\"ahler metric $\mg{\calT_{M,r_0}}$ on ${\calT_{M,r_0}}$ is flat.
 
 For the case $K\neq0$, the metric satisfies:
 \begin{enumerate}
  \item[(i)] if $m\neq2$, the holonomy is ${\mathrm{U}}(m)$.
  \item[(ii)] if $m=2$, the holonomy is $\SU(2)$; in other words, the metric is of the Calabi-Yau type, this is, non-flat K\"ahler and Ricci-flat.
 \end{enumerate}
\end{teo}
\begin{proof}
The result will be achieved by computing the local holonomy on the zero-section ($r=0$). Applying \cite[Theorem 2.1]{Alb2014d} to the obvious vector bundle and metric with curvature $R^M(z,w)=-K(z\wedge w)$, and weights induced by $\varphi_1=\log\mu$, $\varphi_2=\log\lambda$, we find with respect to the canonical decomposition $H^\na\oplus V$:
\[ R^{\mg{\tm}}(z^h,w^h)=\left[\begin{array}{cc}
                            \sqrt{c_1}R^M & 0 \\ 0 & \frac{1}{\sqrt{c_1}}(R^M)^v
                             \end{array}\right] ,\  R^{\mg{\tm}}(z^v,w^v)=\left[\begin{array}{cc}
    \frac{1}{\sqrt{c_1}}(R^M)^h & 0 \\
    0 & -\frac{K}{c_1\sqrt{c_1}}z^v\wedge w^v
                         \end{array}\right]   \]
and 
\[   R^{\mg{\tm}}(z^h,w^v)=\left[\begin{array}{cc}
                            0& -Q \\  Q^t &0
                             \end{array}\right]      \] 
where $Q$ corresponds with
\[ \mg{\tm}(R^{\mg{\tm}}(z^h,w^v)x^h,y^v) =\frac{K}{2c_1}\sqrt{c_1}\langle z^h,x^h\rangle_{_M}\langle w^v,y^v\rangle_{_M}-\frac{K}{2\sqrt{c_1}}\langle z^v\wedge x^v,w^v\wedge y^v\rangle_{_M}   .  \]
These formulae show $\mg{\tm}$ is indeed flat if $K=0$. 

Now we find that $R^{\mg{\tm}}(z^v,w^v)=\frac{1}{c_1}R^{\mg{\tm}}(z^h,w^h)$, \,$\forall z,w\in TM$. Furthermore, 
\[ R^{\mg{\tm}}(z^h,w^v)x^h=\frac{K}{2}(\langle z^h,x^h\rangle w^v-\langle z^v,w^v\rangle x^v+\langle x^v,w^v\rangle z^v)  .  \]
In an $\mg{\tm}_0$-orthonormal frame $e_i,f_i=Be_i$ of horizontals and verticals we see the latter is
\[ R^{\mg{\tm}}(e_i,f_j)=\frac{K}{2}(f_j\otimes e^i-\delta_{ij}B+f_i\otimes e^j) .  \]
These maps are all linearly independent when we restrict to $i\leq j,\ m\neq2$. Indeed their \textit{trace} is a multiple of $2B-mB$, since $B=\sum f_k\otimes e^k$. 

Finally, if $m\neq2$, the space of skew-symmetric curvature tensors is spanned by the two linearly independent subspaces of endomorphisms found above. Their number clearly adds to $m^2$, this is, to $\dim\uni(m)$. Recall we already know the holonomy algebra is contained in $\uni(m)$. If $m=2$, then the holonomy algebra has dimension $m^2-1=3$ and so it is $\sul(2)$.
\end{proof}
It is agreed the term \textit{Calabi-Yau} is reserved for compact manifolds.

One is questioned if the above metric, in case $m=2$, agrees with the well-known Eguchi-Hanson metric on $\calT_{S^2}$, which in turn is the Stenzel metric.
\begin{teo}
  $\mg{\calT_{S^2}}$ coincides with the Eguchi-Hanson or Stenzel metric on $\calT_{S^2}$.
\end{teo}
\begin{proof}
Let us recall the canonical $\mg{\tm}_0$-orthonormal coframe $\alpha_0,\alpha_1,\alpha_2$ on the unit tangent sphere bundle $S{S^2}$ of $S^2$. $\alpha_0,\alpha_1$ are horizontal, $\alpha_0$ is dual to $B^\mathrm{ad}\xi$ and $\alpha_1=\alpha_2\circ B$. Now $TS^2\setminus0$ identifies with $\SO(3)\times(0,+\infty)$, so we find for, say $K=1,c_1=1$, our metric \eqref{theKahlermetric} can be written as
\[  \mg{\tm}=\sqrt{1+r^2}(\alpha_0^2+\alpha_1^2)+\frac{1}{\sqrt{1+r^2}}(\dx r)^2+\frac{r^2}{\sqrt{1+r^2}}\alpha_2^2 . \]
Letting $r=\sinh\tau$ we get
\[  \mg{\tm}=\cosh\tau\,(\alpha_0^2+\alpha_1^2)+\cosh\tau\,(\dx\tau)^2+\sinh\tau\tanh\tau\,\alpha_2^2 . \]
This is the formula for the referred Stenzel metric in the last page of \cite{Sten}. A formula which identifies the metric with the Eguchi-Hanson metric in dimension $m=2$.
\end{proof}
Notice we have given a new proof that the Stenzel metric has $\SU(2)$ holonomy, because the last result does not recur to the specific complex quadric in $\C^3$ which identifies with $\calT_{S^2}$.

We remark also $\calT_{H^2,1}$ may be described analogously with the elliptic trigonometric functions.

It is still an open question how to find the complex coordinates of the space $\tm$ in general. Complex geo\-me\-try properties, like pseudo-convexity, holomorphic completeness and the derived functors from holomorphic to real-base structures, are quite non-trivial issues if the fibres are non-compact and non-complex. Not only the sheaf of germs of holomorphic functions $\calO$ is unknown, also the sheaves $R^q\pi_*\calO$ are very difficult to grasp.

We have the following result, independent of $\lambda,\mu$ in \eqref{themetric}.
\begin{prop}
 Given any two functions $f_1,f_2$ on $\tm$, every real $m$-plane $P$, or $JP$, of the form
\begin{equation}
 P_{f_1,f_2}=\{f_1x^h+f_2x^v:\ x\in TM\}
\end{equation}
is an $\R$-Lagrangian tangent distribution. In particular, $H^\na$ and $V$ are $\R$-Lagrangian distributions and hence the zero-section and the fibres of $\tm$ are totally geodesic $\R$-Lagrangian submanifolds.
\end{prop}
\begin{proof}
For any $x,x_1\in TM$, we have
\begin{eqnarray*}
 \omega(f_1x^h+f_2x^v,f_1x^h_1+f_2x^v_1)&=& \mg{\tm}(f_1\frac{\mu}{\lambda} x^v-f_2\frac{\lambda}{\mu} x^h,f_1x^h_1+f_2x^v_1) \\
 &=& (\lambda^2\frac{\mu}{\lambda}-\mu^2\frac{\lambda}{\mu})f_1f_2g_{_M}(x,x_1)\ = \ 0.
\end{eqnarray*}
For the Riemannian geometry of the fibres or the zero-section cf. \cite{Alb2014d}.
\end{proof}

In the study of fibre bundles with holomorphic structures we have thus three classes of $\C$-analytic spaces with three distinct features. Firstly, we have the many holomorphic vector bundles over a complex manifold, with all objects holomorphic. Secondly, we have the Riemannian twistor bundle over a given Riemannian manifold or the symplectic twistor bundle, with fibre the Siegel domain, over a given symplectic manifold endowed with a symplectic connection. In twistor theory the bundle projection is not holomorphic, since the base may not even be complex. But the fibres are always complex submanifolds. Now, in the third remarkable class, we have $\tm,\mg{\tm},J$, an invariantly defined Hermitian structure, yet a real fibre bundle with fibres just real submanifolds, well-defined even for $m$ odd.

\vspace{3mm}
\begin{center}
\textbf{5. ON A RIEMANN SURFACE}
\end{center}
\setcounter{section}{5}

We start with a remark on a complex integrable system question. Let us consider a real $2n$-manifold $Y$ and its complexified tangent bundle. Suppose we are given $n$ linearly independent commuting vector fields $\Upsilon_i,\ i=1,\ldots,n$, which span $TY^c:=TY\otimes_\R\C$ over $\C$ together with their conjugates. Then these $\Upsilon_1,\ldots,\Upsilon_n$ generate the distribution of $(0,1)$-vector fields of an almost-complex structure $J$ on $Y$. Indeed, we may see $J$ is real and satisfies $J^2=-1$. Also $J$ is integrable, since $[\Upsilon_i,\Upsilon_j]=0,\ \forall i,j$, applying Newlander-Nirenberg Theorem. In general, however, the vector fields $\Upsilon_i$ are not a $(0,1)$-\textit{totally commuting complex frame field}. By this we mean a frame of commuting $(0,1)$-vector fields commuting also with their conjugates. Now, how to find such a frame, that is the question. Complex charts do exist and so a solution exists. But it is not easier to solve such problem independently.

For example, over the 2-plane disk $\{z\in\R^2:\ |z|<1\}$, let $\Upsilon=\overline{z}\partial_z+\partialoverlinez$. Then $\Upsilon$ and $\overline{\Upsilon}$ are linearly independent and hence $\Upsilon$ is $(0,1)$ for an integrable complex structure $J$. We have $[\Upsilon,\overline{\Upsilon}]=\overline{z}\partialoverlinez-z\partial_z$. On the other hand, on some open subset, a $J$-complex chart is given by $\phi(z,\overline{z})=\overline{z}^2-2z$ and hence the desired solution is $\papa{ }{\overline{\phi}}$. Any $(0,1)$-totally complex field is both a multiple $f\Upsilon$, for some $\cinf{}$ function $f$, and a holomorphic multiple of $\papa{ }{\overline{\phi}}$. In particular, we have the solution $f=\frac{-1}{2(1-|z|^2)}$ (one just has to solve $f\Upsilon(\overline{\phi})=1$). On the other hand, searching merely for $f\in\cinf{}(\C)$ such that $[f\Upsilon,\overline{f\Upsilon}]=0$ corresponds with a solution of the, indeed, more complicated equation
\begin{equation}
 |z|^2f\papa{\overline{f}}{z}+zf\papa{\overline{f}}{\overline{z}}- z\overline{f}\papa{f}{\overline{z}}- \overline{f}\papa{f}{z}+\overline{z}|f|^2=0 .
\end{equation}
Is $f=\frac{1}{1-|z|^2}$, up to a constant factor, the only solution of this equation?

For $n=1$, in general, we recall that given a Riemann surface with complex chart $z$ then finding the complex chart for another complex structure is solved by the Beltrami equation. Locally, every $J$ corresponds to an $\Upsilon=\partialoverlinez-\mu\partial_z$ with $\mu\in\cinf{}(\C)$ such that $|\mu|<1$.

Let us now resume with the geometry of $\tm$ as found in Theorem \ref{IntegrabilidadedeSasakigenemdiscobundle}, where $M$ is a Riemann surface of constant sectional curvature $K=\pm1$. We know how to normalise $K$, so $M$ is essentially $\C\Proj^1$ or $D_1\subset\R^2$.

We may assume to have isothermal coordinates of $M$ and hence take a complex chart $z$ on an open subset $U\subset M$ where the metric is $g_{_M}=\frac{2}{(1\pm|z|^2)^2}\dx z\odot\dx \overline{z}=\frac{4(\dx x^2+\dx y^2)}{(1\pm x^2\pm y^2)^2}$. We thus work on the complexified tangent bundle of $M$; we easily see the Levi-Civita connection is given by (as usual $\dx z=\dx x+\sqrt{-1}\dx y$, dual to $\partial_z=\frac{1}{2}(\partial_x-\sqrt{-1}\partial_y)$, $\partialoverlinez=\overline{\partial_z}$)
\begin{equation}
  \na_z\partial_z=\mp\frac{2\overline{z}}{1\pm|z|^2}\partial_z,\qquad\na_z\partialoverlinez=0 .
\end{equation}
Since the metric is K\"ahler on $M$, we have a complex chart $(z,w)$ of the tangent bundle manifold $\tm$, over the subset $\inv{\pi}(U)$, such that $\pi$ is the $1^{\text{st}}$-projection and the vertical tangent subspace $V^c$ is spanned by $\partial_w,\partialoverlinew$. Admitting we had found the horizontal sub-bundle $H^\na$, we then have the mirror map $B\in\mathrm{End}\,(T\tm)^c$ and the induced connection $\na^*$ respecting the canonical decomposition. Immediately,
since $B\partial_z=\partial_w$,
\begin{equation}
 \na^*\partial_w=B\na^*\partial_z\qquad
 \na^*\partialoverlinew=B\na^*\partialoverlinez .
\end{equation}
The tautological vector field $\xi$ on $\tm$ clearly satisfies
\begin{equation}
\xi_{(z,w)}=w\partial_w+\overline{w}\partialoverlinew .  
\end{equation}
Solving $\na^*_X\xi=0$ in the unknown $X=a_1\partial_z+a_2\partialoverlinez+b_1\partial_w+b_2\partialoverlinew$, we obtain a complex basis of $({H^\na})^c$:
\begin{equation}\label{complexbasisofHcomplex}
 X_1=\partial_z\pm \frac{2w\overline{z}}{1\pm|z|^2}\partial_w,\qquad
 X_2=\partialoverlinez\pm \frac{2\overline{w}z}{1\pm|z|^2}\partialoverlinew
 =\overline{X}_1 .
\end{equation}
Now we have the adjoint of $B$ respecting the canonical decomposition and so, for any given real function $a$ of $(z,w)$, we may consider the almost-complex structure
\begin{equation}
 J=aB-\inv{a}B^\mathrm{ad}.
\end{equation}
\begin{teo}
 The structure $J$ is integrable if and only if with any constant $c_1>0$
 \begin{equation}\label{formulaofa}
  a=\sqrt{c_1\pm\frac{4|w|^2}{(1\pm|z|^2)^2}}  .
 \end{equation}
\end{teo}
\begin{proof}
The previous arguments in Theorem \ref{IntegrabilidadedeSasakigenemdiscobundle} regarding the Nijenhuis tensor apply again. In particular, just one equation decides all. Solving $N(X_1,X_2)=0$, we find
 \begin{equation}
  \begin{cases}
   (1\pm|z|^2)\papa{a}{z}\pm2w\overline{z}\papa{a}{w}=0\\
   (1\pm|z|^2)^2a^2\papa{a}{w}\mp2\overline{w}a
   \pm2z\overline{w}(1\pm|z|^2)\papa{a}{z}+4|w|^2|z|^2\papa{a}{w}=0
  \end{cases}
 \end{equation}
whose \textit{unique} solution is the function $a$ in \eqref{formulaofa}. Notice the first equation yields the last summands in the second to cancel.
\end{proof}
This result is a partial improvement of Theorem \ref{IntegrabilidadedeSasakigenemdiscobundle} since the solution $a$ is not supposed \textit{a priori} to be dependent uniquely of $r=\|\xi\|=\frac{2|w|}{1\pm|z|^2}$. Further generality must follow from a not pre-arranged base metric.

The domain restrictions of the referred Theorem must apply again; we assume them from now on and work with $J$ and $a=\sqrt{c_1\pm r^2}$ defined on the disk bundle $\tm={\calT}_{M,r_0}$. For the negative curvature case, we have $r_0=\sqrt{c_1}$. Notice when $|z|\rr1$ we have $w$ in disk fibres of ray going to infinite in $|\cdot|$-norm.

Let us denote ($\Upsilon_i=X_i+\sqrt{-1}JX_i$)
\begin{equation}
 \Upsilon_1=X_1+\sqrt{-1}a\partial_w, \qquad\Upsilon_2=X_2+\sqrt{-1}a\partialoverlinew  .
\end{equation}
Then $\Upsilon_1,\Upsilon_2$ span the vector bundle $T^{0,1}\tm$ of $-\sqrt{-1}$-eigenvectors of $J$.
\begin{prop}
 $f\in\cinf{\inv{\pi}(U)}(\C)$ is $J$-holomorphic if and only if $\Upsilon_1(f)=0$, $\Upsilon_2(f)=0$.
\end{prop}
All questions are driven into the realm of complex variables, if we could find holomorphic charts. Some remarkable relations are at hand.
\begin{prop}
 Letting $\mathfrak{z}=1\pm|z|^2$, the following are satisfied:
 \begin{equation}\label{expected1}
    \papa{a}{z}=-\frac{4\overline{z}|w|^2}{a\mathfrak{z}^3}\qquad\quad
   \papa{a}{w}=\pm\frac{2\overline{w}}{a\mathfrak{z}^2}\qquad\quad X_1(a)=X_2(a)=0 
\end{equation}
\begin{equation}\label{expected2}
  [X_1,X_2]=\pm\frac{2}{\mathfrak{z}^2}(\overline{w}\partialoverlinew-w\partial_w)= [a\partial_w,a\partialoverlinew]\quad\qquad 
  [X_1,\partialoverlinew]=[X_2,\partial_w]=0  
\end{equation}
\begin{equation}
 [\Upsilon_1,\Upsilon_2]=0\qquad\quad
 [\Upsilon_1,\overline{\Upsilon}_2]=\sqrt{-1}\,2a[\partial_w,X_1]=\pm\sqrt{-1}\,\frac{4a\overline{z}}{\mathfrak{z}}\partial_w
\end{equation}
\begin{equation}
 [\overline{\Upsilon}_1,\overline{\Upsilon}_1]=2[X_1,X_2]=-
 [\overline{\Upsilon}_2,\overline{\Upsilon}_2] .
\end{equation}
\end{prop}
 After checking these simple computations, one may also observe the following. Since the $X_i$ from \eqref{complexbasisofHcomplex} are horizontal and $a$ found in \eqref{formulaofa} depends only of $r$, the right hand side of \eqref{expected1} comes as expected from \eqref{derivativeofr}. The first equation of \eqref{expected2} is the vertical part of \eqref{fundamentalidentity}. More importantly, a $(0,1)$-{totally commuting complex frame} does not appear easily.

\vspace{7mm}

The research leading to these results has received funding from the People Programme (Marie Curie Actions) of the European Union's Seventh Framework Programme (FP7/2007-2013) under REA grant agreement nº PIEF-GA-2012-332209.

\ 

\textsc{R. Albuquerque}

{\texttt{rpa@uevora.pt}}
Centro de Investiga\c c\~ao em Mate\-m\'a\-ti\-ca e Aplica\c c\~oes

Universidade de \'Evora, Portugal

\end{document}